\documentclass[11pt, a4paper, notitlepage]{article}

\usepackage{amsmath, latexsym, amsthm, amsfonts, amssymb} 
\usepackage{graphicx} 
\usepackage{epsfig}
\usepackage{varioref}
\usepackage{natbib} 
\usepackage{enumitem}

\numberwithin{equation}{section}
\theoremstyle{plain}
\newtheorem{thm}{Theorem}[section]

\newtheorem{cor}{Corollary}[section]
\newtheorem{lem}{Lemma}[section]

\newtheorem{met}{Method}

\theoremstyle{definition} 
\newtheorem{exam}{Example}[section]

\newtheorem{rem}{Remark}[section]

\makeindex

\newcommand{\beao}{\begin{eqnarray*}}
\newcommand{\eeao}{\end{eqnarray*}\noindent}

\newcommand{\beam}{\begin{eqnarray}}
\newcommand{\eeam}{\end{eqnarray}\noindent}

\newcommand{\beqq}{\begin{equation}}
\newcommand{\eeqq}{\end{equation}\noindent}

\newcommand{\bce}{\begin{center}}
\newcommand{\ece}{\end{center}}

\newcommand{\barr}{\begin{array}}
\newcommand{\earr}{\end{array}}
\DeclareMathOperator*{\argmin}{arg\,min}

\begin{document}

\title {Semiparametrically Efficient Estimation of Euclidean Parameters under Equality Constraints}
\author{
Chris A.J. Klaassen and Nanang Susyanto,\\
Korteweg-de Vries Institute for Mathematics \\
University of Amsterdam\\
P.O. Box 94248, 1090 GE Amsterdam, The Netherlands\\
 email: c.a.j.klaassen@uva.nl\\
 email: n.susyanto@uva.nl}

\maketitle

\begin{abstract}
Assume a (semi)parametrically efficient estimator is given of the Euclidean parameter in a (semi)parametric model. A submodel is obtained by constraining this model in that a continuously differentiable function of the Euclidean parameter vanishes. We present an explicit method to construct (semi)parametrically efficient estimators of the Euclidean parameter in such equality constrained submodels and prove their efficiency. Our construction is based solely on the original efficient estimator and the constraining function.

Only the parametric case of this estimation problem and a nonparametric version of it have been considered in literature.
\end{abstract}

\section{Introduction}\label{I}

Let $X_1,\dots,X_n$ be i.i.d. copies of $X$ taking values in the measurable space $({\cal X},{\cal A})$ in a regular semiparametric model with Euclidean parameter $\theta \in \Theta,$ where $\Theta$ is an open subset of $\mathbb{R}^k.$ We denote this semiparametric model by
\begin{equation}\label{model}
{\cal P}=\left\{P_{\theta, G}\ :\ \theta \in \Theta,\  G \in {\cal
G} \right\}.
\end{equation}
Typically, the nuisance parameter space $\cal G$ is a subset of a
Banach or Hilbert space. If this space is finite dimensional, we are dealing with a parametric model.

We assume an asymptotically efficient estimator $\hat \theta_n=\hat \theta_n(X_1,\dots,X_n)$ is given of the parameter of interest $\theta,$ which under regularity conditions means that
\begin{equation}\label{defefficientestimator}
\sqrt n\left(\hat \theta_n-\theta-\frac1n \sum\limits_{i=1}^n {\tilde \ell}(X_i;\theta,G,{\cal P})\right) \rightarrow_{P_{\theta,G}} 0
\end{equation}
holds. Here ${\tilde \ell}(\cdot;\theta,G,{\cal P})$ is the efficient influence function for estimation of $\theta$ within $\cal P$ and
\begin{equation}\label{matrixboundP}
I^{-1}(\theta,G,{\cal P})=\int_{\cal X}{\tilde \ell}(x;\theta,G,{\cal P}){\tilde \ell}^T\!(x;\theta,G,{\cal P})dP_{\theta,G}(x)
\end{equation}
is the information bound, which corresponds to the efficient information matrix $I(\theta,G,{\cal P}).$

Quite frequently the elements of the parameter of interest $\theta=(\theta_1,\dots,\theta_k)$ are not mathematically independent but satisfy $d$ functional relationships $S_i(\theta) =0,\ i=1,\dots,d,$ with $d<k.$ Formally, this can be described as
\begin{equation}\label{restrictionfunctionS}
S(\theta) =0 ,\quad \theta \in \Theta,
\end{equation}
where $S$ is a function from $\mathbb{R}^k$  to $\mathbb{R}^d.$ We will assume that the $d \times k$ Jacobian matrix ${\dot S}(\cdot)$ exists, is continuous in $\theta$ on $\Theta,$ and has full rank $d.$
Thus, we have constrained the semiparametric model $\cal P$ to a semiparametric
submodel of it, namely
\begin{equation}\label{submodel}
{\cal Q}=\left\{P_{\theta, G}\, :\, S(\theta)=0,\ \theta \in \Theta,\ G \in {\cal G} \right\}.
\end{equation}
Given the constraint $S(\theta)=0$,
we will adapt the semiparametrically efficient estimator $\hat \theta_n$ of
$\theta$  within $\cal P$ in such a way that the adapted estimator is semiparametrically efficient within the constrained model $\cal Q$. Of course, it has to have at least as small asymptotic variance as the original estimator $\hat\theta_n$ and to be at least as close to the true value stochastically.

Efficient estimation of Euclidean parameters under equality constraints for {\em nonparametric} models has been studied in \cite{Levit75}, \cite{KoshevnikLevit76}, \cite{Haberman84}, \cite{Sheehy88}, in Example 1.3.6, 3.2.3, and 3.3.3 of Bickel et al. (1993), henceforth called \cite{BKRW93}, in \cite{MuellerWefelmeyer02}, and in \cite{BroniatowskiKeziou12}. In \cite{BKRW93} nonparametric models under equality constraints are called constraint defined models. Let the semiparametric model $\cal P$ be embedded into a nonparametric model $\tilde{\cal P}$ and let the map $\nu : {\tilde{\cal P}} \to {\mathbb R}^k$ be such that $\nu(P_{\theta,G}) = \theta$ holds for all $P_{\theta, G} \in {\cal P}.$ In view of ${\cal P} \subset {\tilde{\cal P}}$ estimation of $\nu(P)$ within $\cal P$ is easier than within ${\tilde{\cal P}}.$ This relation between these models also holds under the equality constraint $S(\nu(P))= 0.$ Consequently, the results for nonparametric models under constraints are not directly applicable to our semiparametric situation.

For the constrained parametric estimation problem so-called restricted maximum likelihood estimators have been studied. \cite{AitchisonSilvey58} have used Lagrange multipliers with an iterative computation method. An alternative iterative construction has been proposed by \cite{Jamshidian04}, who also presents a long list of examples of constrained parametric estimation problems. To prove efficiency of these restricted maximum likelihood estimators additional regularity conditions are needed. Our method does not need these additional conditions, provided an efficient estimator for the original unconstrained parametric model is given. Finite sample Cram\'er-Rao bounds for the constrained parametric case have been derived by e.g. \cite{GormanHero90}, \cite{Marzetta93}, and \cite{Stoica98}.

To the best of our knowledge the semiparametric version of the topic of the present paper has not been studied in literature yet.

Estimation of the Euclidean parameters constrained by equalities is quite different from estimation of parameters constrained by {\em inequalities}. A comprehensive treatment of the latter estimation problems may be found in \cite{Eeden06}.

If $\cal Q$ can be reparametrized as
\begin{equation}\label{Qreparametrized}
{\cal Q}=\left\{P_{f(\nu), G}\, :\, \nu \in N,\ G \in {\cal G} \right\},
\end{equation}
where $N$ is open and $f\,:\, N \to \Theta$ is injective and continuously differentiable with full rank Jacobian, then $\nu$ can be estimated semiparametrically efficiently as in \cite{KlaassenSusyanto15} and, as noted there, $\theta$ can be estimated efficiently as well by applying $f(\cdot)$ to the efficient estimator of $\nu.$ However, it may be hard or even impossible to find such a reparametrization. A simple, formal example is estimation of the mean vector of a bivariate normal distribution where it is known that this mean vector lies on the unit circle. The unit circle cannot be parametrized as in (\ref{Qreparametrized}) with $N$ open and $f(\cdot)$ continuous and injective. Indeed, assume $\nu_n \in N$ converge to a point at the boundary of $N.$ Then $f(\nu_n)$ converge to a point on the unit circle $f(\nu_0),$ say, with $\nu_0 \in N.$ But by the continuity of $f(\cdot)$ this implies that there exist a point in $N$ close to the boundary of $N$ and a point in $N$ close to $\nu_0$ that are mapped on the same point of the circle by $f(\cdot),$ which contradicts its injectivity. On the other hand there are submodels $\cal Q$ of the type (\ref{Qreparametrized}) that cannot be viewed as a submodel of the type (\ref{submodel}). Again consider estimation of the mean vector of a bivariate normal distribution where it is known now that this mean vector lies on the unit circle with one point removed. This unit circle with one point removed can be parametrized as in (\ref{Qreparametrized}) with $f(\cdot)$ continuous and $N$ open, but it cannot be described via (\ref{restrictionfunctionS}), since the preimage of the closed set $\{0\}$ under a continuous function $S(\cdot)$ has to be closed and  the unit circle with one point removed is not. In the present paper, everything will be done directly to the original parameter subject to equality constraints without reparametrizing it.

The outline of the paper is as follows. In Section \ref{EIFP}, we will present a lower bound to the efficient information bound for estimating the parameter of interest within the constrained model $\cal Q.$ This lower bound will be formulated in terms of the efficient information bound of the original model $\cal P$ and the Jacobian of the constraining function $S(\cdot).$ An explicit estimator that is efficient within the constrained model, will be given in Section \ref{EEEC}. It attains the lower bound from Section \ref{EIFP}, which shows that both this information bound and the estimator are efficient within the constrained model. Examples are discussed in Section \ref{E}. Our conclusions are presented in Section \ref{conclusion}.

\section{Efficient Influence Functions and Projection}\label{EIFP}

In the situation of Section \ref{I} we denote the so-called {\em efficient score function} for $\theta$ by
\begin{equation}\label{efficientscorefunction}
\ell^*(\cdot;\theta,G,{\cal P}) = I(\theta,G,{\cal P}){\tilde \ell}(\cdot;\theta,G,{\cal P}).
\end{equation}
We will restrict attention to regular semiparametric models for which at every $P_0 = P_{\theta_0, G_0} \in {\cal P}$ the parameter $\theta$ is pathwise differentiable, the tangent space $\dot{\cal P}$ is the sum of the tangent space ${\dot{\cal P}}_1$ for $\theta$ and the tangent space ${\dot{\cal P}}_2$ for $G$, and the efficient score function $\ell^*(\cdot;\theta,G,{\cal P})$ for $\theta$  is the projection of the (ordinary) score function ${\dot \ell}(\cdot;\theta,G,{\cal P})$ for $\theta$ on the orthocomplement of ${\dot{\cal P}}_2$ within $\dot{\cal P}$ in the sense of componentwise projection within $L_2^0(P_0)=\{f \in L_2(P_0):E_{P_0}f(X)=0\};$ for details see Chapter 3 of \cite{BKRW93} and Chapter 25 of \cite{Vaart98}.

By Proposition 3.3.1 of \cite{BKRW93} the efficient influence function ${\tilde \ell}(\cdot;\theta,G,{\cal Q})$ for $\theta$ within the submodel $\cal Q$ can be obtained by projecting the efficient influence function ${\tilde \ell}(\cdot;\theta,G,{\cal P})$ for $\theta$ within $\cal P$ onto the tangent space $\dot{\cal Q}$ of $\cal Q$ or onto an appropriate subspace of this tangent space.

Let $\{ \theta_\eta \,:\, \theta_\eta \in {\mathbb R}^k,\ \eta \in {\mathbb R},\ |\eta| < \epsilon\}$ for sufficiently small $\epsilon >0$ be a path through $\theta_0$ in ${\mathbb R}^k$ in the direction $r \in {\mathbb R}^k,$ which means that $|\theta_\eta - \theta_0 - \eta r | = o(|\eta|).$ If this path satisfies $S(\theta_\eta)=0,\, |\eta|<\epsilon,$ then the differentiability of $S(\cdot)$ at $\theta_0$ implies $|S(\theta_\eta) - S(\theta_0) - \eta {\dot S}(\theta_0)r| = o(|\eta|),$ meaning $|\eta {\dot S}(\theta_0)r| = o(|\eta|),$ and hence ${\dot S}(\theta_0)r = 0.$ In other words, such a path within the parameter set $\{ \theta \,:\, S(\theta) =0,\, \theta \in {\mathbb R}^k\},$ has a direction $r$ at $\theta_0$ that belongs to the orthocomplement of the $d$-dimensional linear space within ${\mathbb R}^k$ spanned by the $d$ row vectors of the Jacobian matrix ${\dot S}(\theta_0).$
In fact, to each element of this orthocomplement $[{\dot S}(\theta_0)]^\perp$ corresponds such a path, as is proved in detail in Appendix \ref{EPGD} with the help of the implicit function theorem.

 With $P_0 \in {\cal Q}$ let $L$ be a $k\times(k-d)$-matrix, whose columns span this $(k-d)$-dimensional orthocomplement. Since $\cal P$ is a regular semiparametric model, the parametric submodel ${\cal P}_1 = \{P_{\theta,G_0}\,:\, \theta \in \Theta \}$ is regular. With $s(\theta)$ denoting the square root of the density of $P_{\theta,G_0}$ with respect to an appropriate dominating measure $\mu,$ this regularity implies
\begin{equation}
|| s(\theta_\eta) - s(\theta_0) - \tfrac 12 s(\theta_0) (\theta_\eta - \theta_0)^T {\dot \ell}(\theta_0)||_\mu = o(|\theta_\eta - \theta_0|),\quad \theta_\eta \to \theta_0,
\end{equation}
where $||\cdot||_\mu$ is the norm of $L_2(\mu)$ and ${\dot \ell}(\theta_0) = {\dot \ell}(\cdot;\theta_0,G_0,{\cal P})$ is the score function for $\theta$ at $\theta_0$; cf. Definition 2.1.1 and formula (2.1.4) of \cite{BKRW93}. For a path $\{ \theta_\eta \,:\, \theta_\eta \in {\mathbb R}^k,\ \eta \in {\mathbb R},\ |\eta| < \epsilon\}$ with direction $r$ at $\theta_0$ as above, this implies
\begin{equation}
|| s(\theta_\eta) - s(\theta_0) - \tfrac 12 \eta s(\theta_0) r^T{\dot \ell}(\theta_0) ||_\mu = o(|\eta|),\quad \eta \to 0.
\end{equation}
Consequently, we are dealing here with a 1-dimensional regular parametric model with score function
$r^T{\dot \ell}(\theta_0)$ for $\eta$ at $\eta=0.$ It follows that the closed linear span $\left[ L^T \dot{\ell}(\theta_0) \right]$ of all such score functions $r^T{\dot \ell}(\theta_0)$ is the tangent space ${\dot{\cal Q}}_1$ of ${\cal Q}_1 = \{P_{\theta,G_0}\,:\, S(\theta)=0,\ \theta \in \Theta \}$ at $P_0.$ This implies that the tangent space $\dot{\cal Q}$ of $\cal Q$ at $P_0$ contains both $\left[ L^T \dot{\ell}(\theta_0) \right]$ and ${\dot{\cal P}}_2.$ Writing $\ell^*(\theta_0)$ for $\ell^*(\cdot; \theta_0,G_0,{\cal P})$ we have  for every tangent $t \in {\dot{\cal P}}_2$
\begin{equation}\label{tangents}
r^T{\dot \ell}(\theta_0) + t = r^T \ell^*(\theta_0) + t + r^T\left({\dot \ell}(\theta_0) - \ell^*(\theta_0)\right).
\end{equation}
Since $\ell^*(\theta_0)$ is the componentwise projection of ${\dot \ell}(\theta_0)$ on the orthocomplement of ${\dot{\cal P}}_2,$ each component of ${\dot \ell}(\theta_0) - \ell^*(\theta_0)$ belongs to ${\dot{\cal P}}_2$ and we obtain from (\ref{tangents})
\begin{equation}\label{tangentspaces}
\dot{\cal Q} \supset \left[ L^T \dot{\ell}(\theta_0) \right] + \dot{\cal P}_2 = \left[ L^T \ell^*(\theta_0) \right] + \dot{\cal P}_2 \supset \left[ L^T \ell^*(\theta_0) \right].
\end{equation}
Taking $\theta = \theta_0$  in formula (\ref{efficientscorefunction}) and suppressing $\theta_0$ and $\cal P$ from the notation we rewrite (\ref{tangentspaces}) as
\begin{equation}\label{tangentspaces2}
\dot{\cal Q} \supset \left[ L^T \dot{\ell} \right] + \dot{\cal P}_2 = \left[ L^T \ell^* \right] + \dot{\cal P}_2 \supset \left[ L^T \ell^* \right] = \left[ L^T I {\tilde \ell} \right].
\end{equation}
We shall denote the componentwise inner product within $L_2^0(P_0)$ by $<\cdot,\cdot>_0$ and the projection within $L_2^0(P_0)$ of the efficient influence function $\tilde \ell$ into $\left[ L^T I {\tilde \ell} \right]$ by
\begin{equation}\label{projection}
\Pi_0\left( {\tilde \ell}\ |\, \left[ L^T I {\tilde \ell} \right] \right) = A L^T I {\tilde \ell},
\end{equation}
where $A$ is a $k \times (k-d)$-matrix. Since ${\tilde \ell} - \Pi_0\left( {\tilde \ell}\ |\, \left[ L^T I {\tilde \ell} \right] \right)$ has to be orthogonal to $\left[ L^T I {\tilde \ell} \right],$ i.e., since
\begin{eqnarray}\label{orthogonality}
\left< {\tilde \ell} - A L^T I {\tilde \ell} \,,\,{\tilde \ell}^T I L \right>_0 = I^{-1} I L - A L^T I I^{-1} I L = 0
\end{eqnarray}
holds, we have
\begin{equation}\label{projection2}
\Pi_0\left( {\tilde \ell}\ |\, \left[ L^T I {\tilde \ell} \right] \right) = L \left(L^T I L \right)^{-1} L^T I {\tilde \ell} .
\end{equation}
In order to write this projection in terms of ${\dot S} = {\dot S}(\theta_0)$ we note that according to the Appendix of \cite{KlaassenSusyanto15} $L(L^T I L)^{-1} L^T I + I^{-1} {\dot S}^T ({\dot S} I^{-1} {\dot S}^T )^{-1} {\dot S}$ is the identity map, which implies
\begin{equation}\label{projection3}
\Pi_0\left( {\tilde \ell}\ |\, \left[ L^T I {\tilde \ell} \right] \right) = {\tilde \ell} - I^{-1} {\dot S}^T ({\dot S} I^{-1} {\dot S}^T )^{-1} {\dot S}{\tilde \ell}.
\end{equation}
By Theorem 3.3.2.A of \cite{BKRW93} and formula (3.3.27) in particular, this implies that the limit distribution under $P_0$ of any properly normalized regular estimator of $\theta$ within the submodel $\cal Q$ is the convolution of a normal distribution with mean 0 and covariance matrix
\begin{equation}\label{limitcovariance}
L \left(L^T I L \right)^{-1} L^T = I^{-1} - I^{-1} {\dot S}^T ({\dot S} I^{-1} {\dot S}^T )^{-1} {\dot S}I^{-1}
\end{equation}
and some other distribution. In the next Section we shall construct an estimator of $\theta$ within $\cal Q$ that is asymptotically linear in the influence function from (\ref{projection3}). Consequently, it is asymptotically normal with minimal covariance matrix, i.e.,
\begin{equation}\label{limitbehavior}
\sqrt{n}\left( {\tilde \theta} - \theta_0 \right) \rightarrow_{P_0} {\cal N} \left(0, I^{-1} - I^{-1} {\dot S}^T ({\dot S} I^{-1} {\dot S}^T )^{-1} {\dot S}I^{-1} \right)
\end{equation}
holds.

\section{Efficient Estimator under Equality Constraints}\label{EEEC}

Note that $S(\theta) = S(\theta) - S(\theta_0) = {\dot S}(\theta_0)(\theta - \theta_0) + o(|\theta - \theta_0|)$ holds for $\theta_0$ with $S(\theta_0)=0.$ Since an efficient estimator ${\hat \theta}_n$ within $\cal P$ is asymptotically linear in the efficient influence function ${\tilde \ell}(\cdot; \theta,G,{\cal P}),$ this implies that $S({\hat \theta}_n)$ is asymptotically linear in the influence function ${\dot S}(\theta_0){\tilde \ell}(\cdot; \theta_0,G_0,{\cal P})$ under $\theta_0.$ In order to construct an efficient estimator of $\theta$ within $\cal Q$ we will use this asymptotic linearity.

Our main result reads as follows.
\begin{thm}\label{theoryefficiency}
Consider the regular semiparametric model $\cal P$ and its submodel $\cal Q$ given by \eqref{model} and \eqref{submodel}, respectively. Assume that $S\,:\, {\mathbb R}^k \to {\mathbb R}^d, d < k,$ is continuously differentiable with Jacobian matrix ${\dot S}(\cdot)$ of full rank $d,$ and that the tangent spaces satisfy the conditions mentioned in the first paragraph of Section \ref{EIFP}. Let $X_1,\ldots,X_n$ be i.i.d. with distribution $P \in \cal P$ and suppose that $\hat\theta_n$ is an efficient estimator of the parameter of interest $\theta$ within $\cal P$ based on $X_1,\ldots,X_n$ with efficient influence function $\tilde \ell(\cdot;\theta,G,{\cal P})$ and that ${\hat I}_n$ is a consistent estimator of $I(\theta,G,{\cal P})$ from \eqref{matrixboundP}. Write
\begin{equation}\label{efficientestimator0}
\theta_n^* = \hat\theta_n-{\hat I}_n^{-1}{\dot S}^T\!(\hat\theta_n)\left({\dot S}(\hat\theta_n) {\hat I}_n^{-1}
{\dot S}^T\!(\hat\theta_n) \right)^{-1}S(\hat\theta_n)
\end{equation}
and define
\begin{equation}\label{efficientestimator}
{\tilde \theta}_n = \argmin_{\zeta,\, S(\zeta)=0} \parallel \zeta - \theta_n^* \parallel
\end{equation}
with $\parallel \cdot \parallel$ the Euclidean norm or a topologically equivalent norm.
Then ${\tilde \theta}_n$ efficiently estimates $\theta$ within the submodel $\cal Q$ with efficient influence function
\begin{eqnarray}\label{influencefunctioninQ}
\lefteqn{\tilde\ell(\cdot;\theta,G,{\cal Q}) = \tilde \ell(\cdot;\theta,G,{\cal P})}\\
 && - I^{-1}(\theta,G,{\cal P}) {\dot S}^T\!(\theta) \left({\dot S}(\theta) I^{-1}(\theta,G,{\cal P}) {\dot S}^T\!(\theta) \right)^{-1} {\dot S}(\theta)\tilde \ell(\cdot;\theta,G,{\cal P}) \nonumber
\end{eqnarray}
and hence it satisfies \eqref{limitbehavior}. Furthermore,
\begin{equation}\label{equivalence}
{\sqrt n}({\tilde \theta}_n - \theta_n^*) \rightarrow_{P_{\theta,G}} 0
\end{equation}
holds.
\end{thm}

\begin{proof}
In view of the convolution result proved in Section \ref{EIFP} (cf. (\ref{limitcovariance})) it suffices to show that ${\tilde \theta}_n$ is asymptotically linear in the influence function from \eqref{influencefunctioninQ}, since this yields both sharpness of the convolution bound and efficiency of the estimator.
Fix $\theta_0$ with $S(\theta_0)=0$ and $P_0=P_{\theta_0,G_0} \in {\cal Q},$ and write
\begin{eqnarray}\label{prooflinearity}
\lefteqn{{\sqrt n}\left( \theta_n^* - \theta_0 - \tfrac 1n \sum_{i=1}^n \tilde\ell(X_i;\theta_0,G_0,{\cal Q}) \right) \nonumber} \\
&& = {\sqrt n}\left( {\hat \theta}_n - \theta_0 - \tfrac 1n \sum_{i=1}^n \tilde\ell(X_i;\theta_0,G_0,{\cal P}) \right) \nonumber \\
&& \quad -\, {\hat I}_n^{-1}{\dot S}^T\!(\hat\theta_n) \left({\dot S}(\hat\theta_n) {\hat I}_n^{-1}
{\dot S}^T\!(\hat\theta_n) \right)^{-1} \nonumber \\
&& \qquad \quad \times{\sqrt n} \left(S(\hat\theta_n) - \tfrac 1n \sum_{i=1}^n {\dot S}(\theta_0)\tilde \ell(X_i;\theta_0,G_0,{\cal P}) \right) \\
&& \quad - \left( {\hat I}_n^{-1}{\dot S}^T\!(\hat\theta_n)\left({\dot S}(\hat\theta_n) {\hat I}_n^{-1}
{\dot S}^T\!(\hat\theta_n)\right)^{-1} \right. \nonumber \\
 && \left. \qquad \qquad -  I^{-1}(\theta_0,G_0,{\cal P}) {\dot S}^T\!(\theta_0) \left({\dot S}(\theta_0) I^{-1}(\theta_0,G_0,{\cal P}) {\dot S}^T\!(\theta_0) \right)^{-1} \right){\dot S}(\theta_0) \nonumber \\
&& \qquad \quad \times \tfrac 1{\sqrt n} \sum_{i=1}^n \tilde \ell(X_i;\theta_0,G_0,{\cal P}) \nonumber \\
&& = R_{1,n} - R_{2,n} - R_{3,n}. \nonumber
\end{eqnarray}
The asymptotic linearity of ${\hat \theta}_n$ from \eqref{defefficientestimator} implies that $R_{1,n}$ converges to 0 in probability under $P_0.$ By the central limit theorem the second factor of $R_{3,n}$ is asymptotically normal with mean 0 and covariance matrix $I^{-1}(\theta_0,G_0,{\cal P})$ from \eqref{matrixboundP}. Since ${\dot S}(\cdot)$ is continuous and ${\hat I}_n$ and ${\hat \theta}_n$ are consistent in estimating $I(\theta_0,G_0,{\cal P})$ and $\theta_0,$ respectively,
this implies that $R_{3,n}$ converges to 0 in probability under $P_0$ as well. We also conclude that the first factor of $R_{2,n}$ is bounded in probability. Together with the asymptotic linearity of $S({\hat \theta}_n),$ as noted at the start of this Section, this yields the convergence of $R_{2,n}$ to 0 in probability under $P_0.$

It remains to be shown that \eqref{equivalence} holds. In view of $S(\theta_0)=0$ and Appendix \ref{EPGD} we may parametrize a part of the zero set of $S(\cdot)$ near $\theta_0$ by
\begin{equation}\label{S0}
{\cal S}_0 = \left\{ \theta \,|\, \theta = \theta_0 + L\eta + r(\eta),\ \eta \in H \right\},
\end{equation}
where the $k-d$ columns of the matrix $L$ span the orthocomplement of $[{\dot S}(\theta_0)],$ $r(\eta) = o(\parallel \eta \parallel)$ holds as $\parallel \eta \parallel$ tends to 0, and $H$ is an appropriate neighborhood of 0 within ${\mathbb R}^{k-d}.$ Note that $n^{-1} \sum_{i=1}^n {\tilde \ell}(X_i;\theta_0,G_0, {\cal Q})$ is of the order $O_p(1/{\sqrt n})$ under $P_0$ and takes its values in $[L]$ in view of \eqref{projection2}. Together with \eqref{S0} this shows that there exists a random $k$-vector ${\tilde R}_n = o_p(1/{\sqrt n})$ such that
\begin{equation}\label{Rtilde}
\theta_0 + \tfrac 1n \sum_{i=1}^n {\tilde \ell}(X_i;\theta_0,G_0, {\cal Q}) + {\tilde R}_n \in {\cal S}_0
\end{equation}
holds with probability tending to 1. Because of the definition of ${\tilde \theta}_n,$ the triangle inequality, and the asymptotic linearity of $\theta_n^*$ in the efficient influence function as proved above, this yields
\begin{eqnarray}\label{triangle}
\lefteqn{\parallel {\tilde \theta}_n - \theta_n^* \parallel \nonumber } \\
&& \leq \parallel \theta_0 + \tfrac 1n \sum_{i=1}^n {\tilde \ell}(X_i;\theta_0,G_0, {\cal Q}) + {\tilde R}_n - \theta_n^* \parallel \nonumber \\
&& \leq \parallel \theta_0 + \tfrac 1n \sum_{i=1}^n {\tilde \ell}(X_i;\theta_0,G_0, {\cal Q})- \theta_n^* \parallel + \parallel {\tilde R}_n \parallel \\
&& =  o_p \left( \tfrac 1{\sqrt n} \right), \nonumber
\end{eqnarray}
which proves  \eqref{equivalence}.
\end{proof}

\begin{rem}\label{estimatorI}
Consistent estimators ${\hat I}_n$ of $I(\theta,G,{\cal P})$ may be constructed from ${\hat \theta}_n$ as in Section 4 of \cite{KlaassenSusyanto15}. In regular parametric cases the Fisher information $I(\theta)= I(\theta,G,{\cal P})$ depends on $\theta$ only and is continuous in it. Consequently, ${\hat I}_n = I({\hat \theta}_n)$ is consistent in estimating $I(\theta)$ then.
\end{rem}

\begin{rem}\label{bothestimators}
According to Theorem \ref{theoryefficiency} the estimators $\theta_n^*$ and ${\tilde \theta}_n$ have the same asymptotic performance to first order. However, only ${\tilde \theta}_n$ is guaranteed to be efficient within $\cal Q,$ since $\theta_n^*$ need not be a zero of $S(\cdot).$ In order to compute ${\tilde \theta}_n$ to the desired order of precision one typically needs an iterative numerical procedure, like Newton-Raphson.
\end{rem}

\begin{rem}\label{remarklinear}
Parametrize the linear case by $S(\theta)=R^T(\theta-\alpha)$ with $R$ a $d \times k$-matrix and $\alpha$ a fixed $k$-vector. Now, ${\dot S}(\theta) = R^T$ holds and the estimator from \eqref{efficientestimator} reduces to
\begin{equation}\label{efficientestimatorlinearcase}
\tilde\theta_n=\hat\theta_n-{\hat I}_n^{-1}R\left(R^T {\hat I}_n^{-1}R \right)^{-1}R^T \left(\hat\theta_n - \alpha \right).
\end{equation}
In terms of a $k \times (k-d)$-matrix $L,$ whose columns span the orthocomplement of $[{\dot S}^T\!(\theta)]=[R],$ this estimator may be written as
\begin{equation}\label{efficientestimatorlinearcase2}
\tilde\theta_n=\alpha + L \left(L^T {\hat I}_n L \right)^{-1}L^T {\hat I}_n \left(\hat\theta_n - \alpha \right)
\end{equation}
according to the Appendix of \cite{KlaassenSusyanto15}. Note that this estimator attains the asymptotic information bound
\begin{equation}\label{informationboundlinearcase}
L \left(L^T I(\theta, G, {\cal P}) L \right)^{-1}L^T.
\end{equation}
Comparing their formula (4.18) to \eqref{efficientestimatorlinearcase2} above we note that the approaches of the present paper and of \cite{KlaassenSusyanto15} yield exactly the same estimator in the linear case, although the approaches differ in the general case.
\end{rem}

\begin{rem}
The estimators ${\tilde \theta}_n$ and ${\hat \theta}_n$ are efficient within the models $\cal Q$ and $\cal P,$ respectively. Since $\cal Q$ is a submodel of $\cal P,$ it is easier to estimate $\theta$ within $\cal Q$ than within $\cal P.$ This is visible in the respective limit distributions by comparing \eqref{defefficientestimator} and \eqref{matrixboundP} to \eqref{limitbehavior}. The difference between the two limit covariance matrices is $I^{-1} {\dot S}^T ({\dot S} I^{-1} {\dot S}^T )^{-1} {\dot S}I^{-1},$ which is positive semidefinite because of the nonsingularity of the symmetric information matrix $I,$ the maximum rank of $\dot S,$ and the fact that the inverse of a symmetric positive definite matrix is also symmetric positive definite.

By Theorem \ref{theoryefficiency}, \eqref{projection2}, and \eqref{projection3} we have
\begin{equation}\label{projection4}
{\tilde \theta}_n - \theta_0 = L \left(L^T I L \right)^{-1}L^T I \left( {\hat \theta}_n - \theta_0 \right) + o_p\left( \tfrac 1{\sqrt n} \right).
\end{equation}
This means that ${\tilde \theta}_n - \theta_0$ may be viewed as a projection of ${\hat \theta}_n - \theta_0$ into $[L],$ approximately. In other words, $\tilde\theta_n$ tends to be closer to the true value $\theta_0$ than ${\hat \theta}_n$ in the metric induced by $I.$
\end{rem}

\section{Examples}\label{E}
Our construction of (semi)parametrically efficient estimators will be illustrated in this section by some examples, all of which have been discussed also in Section 5 of the companion paper \cite{KlaassenSusyanto15}.

\begin{exam}\label{locationscale}\textbf{Coefficient of variation known}

Let $g(\cdot)$ be an absolutely continuous density on $({\mathbb R}, \cal{B})$ with mean 0, variance 1, distribution function $G,$ and derivative $g'(\cdot),$ such that $\int [1+x^2] (g'/g(x))^2 g(x) dx$ is finite. Consider the location-scale family corresponding to $g(\cdot).$ Let there be given efficient estimators ${\bar \mu}_n$ and ${\bar \sigma}_n$ of $\mu$ and $\sigma,$ respectively, based on $X_1, \dots, X_n,$ which are i.i.d. with density $\sigma^{-1}g((\cdot - \mu)/\sigma).$ By $I_{ij}$ we denote the element in the $i$the row and $j$th column of the matrix $I= \sigma^2I(\theta,G,{\cal P}),$ where the Fisher information matrix $I(\theta,G,{\cal P})$ is as defined in (\ref{matrixboundP}) with $\theta =(\mu,\sigma)^T.$ Some computation shows $I_{11}= \int (g'/g)^2 g, I_{12}=I_{21}=\int x(g'/g(x))^2 g(x) dx,$ and $I_{22}= \int [xg'/g(x) + 1]^2 g(x) dx$ exist and are finite; cf. Section I.2.3 of \cite{Hajek67}.

We consider the submodel with the coefficient of variation $\sigma/\mu$ known to be equal to a given constant $c.$ We may put this constraint in a linear form by choosing $S(\theta) = c\theta_1 - \theta_2.$ By Remark \ref{remarklinear}  and Example 5.1 of \cite{KlaassenSusyanto15} this implies that the efficient estimator $\tilde{\theta}_n$ of $\theta$ within the constraint model $\cal Q$ from Theorem \ref{theoryefficiency} equals
\begin{equation}\label{cknowngeneral1}
\tilde{\theta}_n = (\hat{\mu}_n, c \hat{\mu}_n)^T
\end{equation}
with
\begin{equation}\label{cknowngeneral2}
\hat{\mu}_n = \left(I_{11} + 2cI_{12} + c^2 I_{22}\right)^{-1}\left[\left(I_{11} + cI_{12}\right) \bar{\mu}_n + \left(I_{12} + cI_{22}\right) \bar{\sigma}_n \right].
\end{equation}
Similar relations hold for the symmetric and normal cases as discussed in Example 5.1 of \cite{KlaassenSusyanto15}.
Note that one gets another, but still efficient estimator of $\theta,$ if one formulates the constraint in a nonlinear way. Choosing e.g. $S(\theta) = \theta_2/\theta_1 - c,$ we arrive by Theorem \ref{theoryefficiency} at $\theta_n^* = (\mu_n^*, \sigma_n^*)^T,$ where straightforward computations with ${\bar c}_n = {\bar \sigma}_n/{\bar \mu}_n$ yield
\begin{eqnarray}\label{cknowngeneral3}
\lefteqn{\mu_n^* = \left(I_{11} + 2{\bar c}_n I_{12} + {\bar c}_n^2 I_{22}\right)^{-1} } \\
&& \qquad \qquad \qquad \left[\left(I_{11} + \{2{\bar c}_n -c\}I_{12}\right) \bar{\mu}_n + \left(I_{12} + \{2{\bar c}_n -c\}I_{22}\right) \bar{\sigma}_n \right] \nonumber
\end{eqnarray}
and
\begin{eqnarray}\label{cknowngeneral3}
\lefteqn{\sigma_n^* = \left(I_{11} + 2{\bar c}_n I_{12} + {\bar c}_n^2 I_{22}\right)^{-1} } \\
&& \qquad \qquad \qquad \left[\left(cI_{11} + c {\bar c}_n I_{12}\right) \bar{\mu}_n + \left({\bar c}_n I_{12} + {\bar c}_n^2 I_{22}\right) \bar{\sigma}_n \right]. \nonumber
\end{eqnarray}
Indeed, this estimator is asymptotically equivalent to the one from (\ref{cknowngeneral1}), but the corresponding coefficient of variation does not equal $c.$ The projection from (\ref{efficientestimator}) of $\theta_n^* = (\mu_n^*, \sigma_n^*)^T$ yields ${\tilde \theta}_n = ({\tilde \mu}_n, c{\tilde \mu}_n)^T$ with
\begin{eqnarray}\label{cknowngeneral3}
\lefteqn{{\tilde \mu}_n = \left(I_{11} + 2{\bar c}_n I_{12} + {\bar c}_n^2 I_{22}\right)^{-1} } \\
&& \qquad  \left[\left(I_{11} + \tfrac {2{\bar c}_n -c + c^2{\bar c}_n}{1+c^2} I_{12}\right) \bar{\mu}_n + \left(\tfrac{1+c {\bar c}_n}{1+c^2} I_{12} + \tfrac {2{\bar c}_n - c + c {\bar c}_n^2}{1+c^2} I_{22}\right) \bar{\sigma}_n \right], \nonumber
\end{eqnarray}
which is asymptotically equivalent to $\hat{\mu}_n,$ but differs from it.
\end{exam}

\begin{exam}\label{Gaussiancopula}\textbf{Exchangeable Gaussian copula model}

Let $$\mathbf{X}_1=(X_{1,1},\ldots,X_{1,m})^T,\ldots,\mathbf{X}_n=(X_{n,1},\ldots,X_{n,m})^T$$ be i.i.d. copies of $\mathbf{X}=(X_1,\ldots,X_m)^T$. For $i=1,\ldots,m$, the marginal distribution function of $X_i$ is continuous and will be denoted by $F_i.$ It is assumed that $(\Phi^{-1}(F_1(X_1)), \dots, \Phi^{-1}(F_m(X_m)))^T$ has an $m$-dimensional normal distribution with mean 0 and positive definite correlation matrix $C(\theta),$ where $\Phi$ denotes the one-dimensional standard normal distribution function. Here the parameter of interest $\theta$ is the vector in $\mathbb{R}^{m(m-1)/2}$ that summarizes all correlation coefficients $\rho_{rs},\, 1\leq r < s \leq m$. We will set this general Gaussian copula model as our semiparametric starting model $\cal P$, i.e.,
\begin{equation}\label{generalGaussiancopula}
{\cal P}=\{P_{\theta,G}\ :\ \theta=(\rho_{12},\dots,\rho_{(m-1)m})^T\ , G =(F_1(\cdot),\dots,F_m(\cdot))\in {\cal G}\}.
\end{equation}
As argued in \cite{KlaassenSusyanto15} the Van der Waerden or normal scores rank correlation coefficient
\begin{equation}\label{normalscores}
\hat\rho_{rs}^{(n)}=\frac{\frac1n\sum\limits_{j=1}^{n}{\Phi^{-1}\left(\frac n{n+1}\mathbb{F}_r^{(n)}(X_{j,r})\right)\Phi^{-1}\left(\frac n{n+1}\mathbb{F}_s^{(n)}(X_{j,s})\right)}}{\frac1n\sum\limits_{j=1}^{n}{\left[\Phi^{-1}\left(\frac j{n+1}\right)\right]^2}}
\end{equation}
with $\mathbb{F}_r^{(n)}$ and $\mathbb{F}_s^{(n)}$ being the marginal empirical distributions of $F_r$ and $F_s$, respectively, $1\leq r < s \leq m,$ is a semiparametrically efficient estimator of $\rho_{rs}$ with efficient influence function
\begin{align}
\tilde\ell_{\rho_{rs}}(X_r,X_s)&=\Phi^{-1}\left(F_r(X_r)\right)\Phi^{-1}\left(F_s(X_s)\right)\\
&\ \ \ -\tfrac 12 \rho_{rs} \left\{\left[\Phi^{-1}\left(F_r(X_r)\right)\right]^2
+\left[\Phi^{-1}\left(F_s(X_s)\right)\right]^2\right\}\notag.
\end{align}
This means that
\begin{equation}\label{estimatorthetaexc}
\hat\theta_n=(\hat \theta_{n1}, \dots,\hat \theta_{nk})^T = (\hat\rho_{12}^{(n)},\dots,\hat\rho_{(m-1)m}^{(n)})^T, \quad k=m(m-1)/2,
\end{equation}
efficiently estimates $\theta$ within $\cal P$ with efficient influence function
\begin{equation}\label{estimatorinfthetaexc}
\tilde\ell(\mathbf{X}; \theta,G,{\cal P})=(\tilde\ell_{\rho_{12}}(X_1,X_2),\ldots,\tilde\ell_{\rho_{(m-1)m}}(X_{m-1},X_m))^T.
\end{equation}

The submodel
\begin{equation}\label{exchangeablemodel}
{\cal Q} = \left\{ P_{\theta,G}\,:\, \theta = {\bf 1}_k\rho,\ \rho \in (-1/(m-1),1),\ G\in {\cal G}\right\} \subset {\cal P}
\end{equation}
with ${\bf 1}_k$ indicating the vector of ones of dimension $k$ is the exchangeable $m$-variate Gaussian copula model. In this submodel all correlation coefficients have the same value $\rho.$

With $J_k$ the $k\times k$ identity matrix we choose $R = J_k - \tfrac 1k {\bf 1}_k {\bf 1}_k^T$ and $\alpha =0$ in Remark \ref{remarklinear} and obtain
\begin{equation}\label{rhobar}
{\tilde \theta}_n = {\bf 1}_k {\bar \theta}_n = {\bf 1}_k {\bar \rho}_n, \quad {\bar \theta}_n={\bar \rho}_n= \tfrac 1k \sum_{r=1}^{m-1} \sum_{s=r+1}^m {\hat\rho}_{rs}^{(n)} = \tfrac 1k \sum_{j=1}^k {\hat \theta}_{nj},
\end{equation}
as efficient estimator of $\theta$ within submodel $\cal Q.$
\end{exam}

\begin{exam}\label{regression}\textbf{Partial spline linear regression}

 As in Example 5.3 of \cite{KlaassenSusyanto15} the observations are realizations of i.i.d. copies of the random vector $X=(Y,Z^T,U^T)^T$ with $Y, Z,$ and $U$ 1-dimensional, $k$-dimensional, and $p$-dimensional random vectors with the structure
\begin{equation}\label{partialspline}
Y = \theta^T Z + \psi(U) + \varepsilon,
\end{equation}
where the measurement error $\varepsilon$ is independent of $Z$ and $U,$ has mean 0, finite variance, and finite Fisher information for location, and where $\psi(\cdot)$ is a real valued function on ${\mathbb R}^p.$ The distribution function of $Z, U,$ and $\varepsilon$ and the function $\psi(\cdot)$ together constitute the nuisance parameter $G$ whereas $\theta$ is the parameter of interest. \cite{Schick93} presents an efficient estimator of $\theta$ and a consistent estimator of $I(\theta, G, {\cal P})$ in his Theorem 8.1. Consequently our Theorem \ref{theoryefficiency} may be applied directly in order to obtain an efficient estimator of $\theta$ in appropriate submodels $\cal Q$ without our construction of an estimator of $I(\theta, G, {\cal P})$ via characteristic functions. In the linear case of Remark \ref{remarklinear} the parameter of interest $\theta$ within the submodel $\cal Q$ may be reparametrized by $\theta = \alpha + L\nu$ with the vector $\alpha$ and the matrix $L$ known. Now $\nu$ is the parameter of interest and we return to the situation of (\ref{partialspline}) with $X=(Y-\alpha^T Z, Z^T L, U^T)^T.$
\end{exam}

\begin{exam}\label{commonmean}\textbf{Multivariate normal with common mean}

Let $\cal G$ be the collection of nonsingular $k\times k$-covariance matrices and let the parametric starting model be the collection of nondegenerate normal distributions with mean vector $\theta$ and covariance matrix $\Sigma,$
\begin{equation}\label{normalmodel}
{\cal P} = \left\{ P_{\theta,\Sigma} \,:\, \theta \in {\mathbb R}^k,\ \Sigma \in {\cal G} \right\}.
\end{equation}
Efficient estimators of $\theta$ and $\Sigma$ are the sample mean ${\bar X}_n = n^{-1}\sum_{i=1}^n X_i$ and the sample covariance matrix ${\hat \Sigma}_n = (n-1)^{-1} \sum_{i=1}^n (X_i - {\bar X}_n)(X_i - {\bar X}_n)^T,$ respectively. Note that ${\bar X}_n$ attains the finite sample Cram\'er-Rao bound and the asymptotic information bound with $I(\theta, \Sigma, {\cal P}) = \Sigma^{-1}.$

The parametric submodel we consider is
\begin{equation}\label{normalsubmodel}
{\cal Q} = \left\{ P_{\theta,\Sigma} \,:\, \theta \in {\mathbb R}^k,\ \theta= {\bf 1}_k \frac 1k \sum_{j=1}^k \theta_j,\ \Sigma \in {\cal G} \right\},
\end{equation}
in which all marginals of each distribution have the same mean. In view of (\ref{efficientestimatorlinearcase2}) with $L={\bf 1}_k$
\begin{equation}\label{estimatorcommonmeans}
\tilde \theta_n={\bf 1}_k \left(\mathbf{1}_k^T{\hat{\mathbf\Sigma}}_n^{-1} \mathbf{1}_k \right)^{-1}\mathbf{1}_k^T{\hat{\mathbf\Sigma}}_n^{-1}\mathbf{\bar X_n}
\end{equation}
is an efficient estimator of $\theta$ within $\cal Q,$ which attains the asymptotic information bound $\left(\mathbf{1}_k^T \Sigma^{-1} \mathbf{1}_k \right)^{-1}\mathbf{1}_k \mathbf{1}_k^T .$
See also Example 5.4 of \cite{KlaassenSusyanto15}.
\end{exam}

\begin{exam}\label{label}\textbf{Restricted maximum likelihood estimator}

Maximum likelihood estimation of the generalized linear model under linear restrictions on the parameters is done in \cite{Nyquist91} via an iterative procedure using a penalty function. \cite{Kim95} introduce the restricted EM algorithm for maximum likelihood estimation under linear restrictions. \cite{Jamshidian04} compares the performance of the gradient projection and of the expectation-restricted-maximization (ERM) method under linear restrictions. Our approach as described in Remark \ref{remarklinear} with ${\hat \theta}_n$ a(n unrestricted) maximum likelihood estimator avoids such iterative procedures, provided $\hat \theta_n$ can be computed without iterations. Moreover, Theorem \ref{theoryefficiency} is not constrained to linear restrictions.
\end{exam}

\section{Conclusion}\label{conclusion}
In this paper, we have shown that the efficient influence function for estimation of $\theta$ within the semiparametric model
$${\cal Q}=\left\{P_{\theta, G}\ :\ S(\theta)=0, \theta \in \Theta, G\in {\cal G} \right\}$$
can be obtained by projecting the efficient influence function for estimation of $\theta$ within the unconstrained model
$${\cal P}=\left\{P_{\theta, G}\ :\ \theta \in \Theta,\  G \in {\cal G} \right\}.$$
It follows that these influence functions are related by
\begin{eqnarray*}
\lefteqn{\tilde\ell(\theta,G,{\cal Q})} \\
&& = \left(J-I^{-1}(\theta,G,{\cal P}){\dot S}(\theta)^T\left({\dot S}(\theta) I^{-1}(\theta,G,{\cal P}) {\dot S}(\theta)^T \right)^{-1}{\dot S}(\theta)\right)\tilde\ell(\theta,G,{\cal P})
\end{eqnarray*}
and hence the corresponding efficient lower bounds by
\begin{eqnarray*}
\lefteqn{I^{-1}(\theta,G,{\cal Q}) = I^{-1}(\theta,G,{\cal P})}\\
&& -I^{-1}(\theta,G,{\cal P}){\dot S}(\theta)^T\left({\dot S}(\theta) I^{-1}(\theta,G,{\cal P}) {\dot S}(\theta)^T \right)^{-1}{\dot S}(\theta)I^{-1}(\theta,G,{\cal P}).
\end{eqnarray*}
Furthermore, Theorem \ref{theoryefficiency} provides a simple method to upgrade an asymptotically efficient estimator for $\theta$ within the unconstrained model to an efficient estimator within the constrained model.

\appendix

\section{Existence of a Path with a Given Direction}\label{EPGD}
Given a continuously differentiable function $S:\Theta\subset\mathbb{R}^k\mapsto \mathbb{R}^d$ with $k>d$. Define
$$
{\cal M}=\{\theta\in\Theta: S(\theta)=0\}
$$
and let $\theta_0\in{\cal M}$ be such that the Jacobian of the function $S$ at $\theta_0$, say ${\dot S}(\theta_0),$ has full-rank $d$.
Suppose that $r\in\mathbb{R}^k$ with ${\dot S}(\theta_0)r=0.$ We would like to construct a path through $\theta_0$ with direction $r$.

Note that according to the Implicit Function Theorem, there exists an open subset $U \subset \mathbb{R}^{k-d},\, 0 \in U,$ and a unique continuously differentiable function $\phi : U \to \mathcal{M}$ with $\phi(0) = \theta_0$ (usually, called parametrization). If ${\dot \phi}_0$ denotes the Jacobian of the function $\phi$ at 0, then the chain rule gives
$${\dot S}(\theta_0){\dot \phi}_0=0$$
in view of $S(\phi(u))=0$ for every $u\in U.$ This implies
$$im({\dot \phi}_0)\subset{\dot S}(\theta_0)^\bot.$$
Since $dim(im({\dot \phi}_0))=k-d=dim({\dot S}(\theta_0)^\bot)$ we obtain
$$im({\dot \phi}_0)={\dot S}(\theta_0)^\bot.$$
Consequently, the direction $r$ has to belong to $im({\dot \phi}_0),$ which means that there exists a $\nu\in  U$ with ${\dot \phi}_0\nu=r.$
Now define a path
$$\{\theta_\eta\}=\{\phi(\eta \nu)\ : \ \eta\in\mathbb{R},\ |\eta|<\varepsilon\}$$
for sufficiently small $\varepsilon>0,$ which obviously passes through $\theta_0$ because of $\phi(0)=\theta_0$. Then, we have
\begin{align*}
\left|\theta_\eta-\theta_0-\eta r\right|&=\left|\phi(\eta \nu)-\phi(0)-\eta r\right|\\
&\leq\left|\eta{\dot \phi}_0 \nu-\eta r\right|+o(|\eta|)\\
&=o(|\eta|).
\end{align*}

\end{document}